\documentclass[11pt]{article}
\def\thetitle{Spanning trees at the connectivity threshold}

\usepackage{graphicx}

\usepackage{amsmath,amssymb}

\usepackage[usenames,dvipsnames,svgnames,table]{xcolor}
\definecolor{CombinatoricaAqua}{HTML}{00698C}
\definecolor{CombinatoricaBlue}{HTML}{3A3293}
\definecolor{CombinatoricaBrown}{HTML}{66220C}
\definecolor{CombinatoricaRed}{HTML}{DF2A27}
\definecolor{HarvardCrimson}{rgb}{0.6471, 0.1098, 0.1882}

\makeatletter
\let\reftagform@=\tagform@
\def\tagform@#1{\maketag@@@
	{(\ignorespaces\textcolor{CombinatoricaBrown}{#1}\unskip\@@italiccorr)}}
\renewcommand{\eqref}[1]{\textup{\reftagform@{\ref{#1}}}}
\makeatother

\usepackage[backref=page]{hyperref}
\hypersetup{
	unicode,
	pdfencoding=auto,
	pdfauthor={Yahav Alon, Michael Krivelevich and Peleg Michaeli},
	pdftitle={\thetitle},
	pdfsubject={},
	pdfkeywords={},
	colorlinks=true,
	citecolor=CombinatoricaBlue,
	linkcolor=CombinatoricaAqua,
	anchorcolor=CombinatoricaBrown,
	urlcolor=HarvardCrimson}

\usepackage{amsthm}
\usepackage{thmtools}
\usepackage{bbm}
\usepackage{enumitem}
\usepackage[capitalize]{cleveref}
\Crefname{fact}{Fact}{Facts}
\Crefname{claim}{Claim}{Claims}


\declaretheoremstyle[
spaceabove=\topsep, spacebelow=\topsep,
headfont=\color{CombinatoricaBrown}\normalfont\bfseries,
bodyfont=\itshape,
]{thm}
\declaretheoremstyle[
spaceabove=\topsep, spacebelow=\topsep,
headfont=\color{CombinatoricaBrown}\normalfont\bfseries,
bodyfont=\normalfont,
]{dfn}
\declaretheoremstyle[
spaceabove=0.5\topsep, spacebelow=0.5\topsep,
headfont=\color{CombinatoricaBrown}\normalfont\bfseries,
bodyfont=\normalfont,
]{rmk}

\declaretheorem[style=thm,parent=section]{theorem}
\declaretheorem[style=thm,sibling=theorem]{lemma}

\declaretheorem[style=thm,sibling=theorem]{claim}
\declaretheorem[style=thm,sibling=theorem]{proposition}

\usepackage[nobysame,msc-links]{amsrefs}

\BibSpecAlias{misc}{webpage}
\BibSpec{book}{%
	+{}  {\PrintPrimary}                {transition}
	+{,} { \textbf}                     {title} 
	+{.} { }                            {part}
	+{:} { \textit}                     {subtitle}
	+{,} { \PrintEdition}               {edition}
	+{}  { \PrintEditorsB}              {editor}
	+{,} { \PrintTranslatorsC}          {translator}
	+{,} { \PrintContributions}         {contribution}
	+{,} { }                            {series}
	+{,} { \voltext}                    {volume}
	+{,} { }                            {publisher}
	+{,} { }                            {organization}
	+{,} { }                            {address}
	+{,} { \PrintDateB}                 {date}
	+{,} { }                            {status}
	+{}  { \parenthesize}               {language}
	+{}  { \PrintTranslation}           {translation}
	+{;} { \PrintReprint}               {reprint}
	+{.} { }                            {note}
	+{.} {}                             {transition}
	+{}  {\SentenceSpace \PrintReviews} {review}
}

\makeatletter
\renewcommand{\PrintNames@a}[4]{%
	\PrintSeries{\name}
	{#1}
	{}{ and \set@othername}
	{,}{ \set@othername}
	{}{ and \set@othername}
	{#2}{#4}{#3}%
}
\makeatother

\makeatletter
\def\mathcolor#1#{\@mathcolor{#1}}
\def\@mathcolor#1#2#3{%
	\protect\leavevmode
	\begingroup
	\color#1{#2}#3%
	\endgroup
}
\makeatother
\definecolor{Red}{rgb}{0.618,0,0}
\definecolor{Blue}{rgb}{0,0,1}
\definecolor{Green}{rgb}{0,0.298,0}

\usepackage{sectsty}
\sectionfont{\color{CombinatoricaBrown}}
\subsectionfont{\color{CombinatoricaBrown}}
\subsubsectionfont{\color{CombinatoricaBrown}}

\usepackage{soul}
\soulregister\ref7

\usepackage{pifont}
\usepackage{calc}

\usepackage[a4paper]{geometry}
\title{\thetitle}

\makeatletter
\def\namedlabel#1#2{\begingroup
  #2%
  \def\@currentlabel{#2}%
  \phantomsection\label{#1}\endgroup
}
\makeatother

\usepackage{mleftright}
\mleftright

\newcommand{\defn}[1]{{\bfseries #1}}
\newcommand{\ceil}[1]{\lceil #1 \rceil}
\newcommand{\floor}[1]{\lfloor #1 \rfloor}

\newcommand{\eps}{\varepsilon}
\newcommand{\cA}{\mathcal{A}}
\newcommand{\cB}{\mathcal{B}}

\newcommand{\sm}{\setminus}
\newcommand{\es}{\varnothing}

\newcommand{\pr}[0]{\mathbb{P}}
\newcommand{\E}[0]{\mathbb{E}}
\DeclareMathOperator{\Var}{Var}
\newcommand{\whp}[0]{\textbf{whp}}
\newcommand{\Dist}[1]{\mathsf{#1}}
\newcommand{\Bin}{\Dist{Bin}}
\newcommand{\Bernoulli}{\Dist{Bernoulli}}

\author{Yahav Alon \and Michael Krivelevich \and Peleg Michaeli}
\AtEndDocument{\bigskip{\footnotesize%
\par

\textsc{Yahav Alon} \par
\textsc{School of Mathematical Sciences, Tel Aviv University,
  Tel Aviv 6997801, Israel} \par
  \textit{Email:}
  \href{mailto:yahavalo@tauex.tau.ac.il}{\texttt{yahavalo@tauex.tau.ac.il}} \par
  \addvspace{\medskipamount}

\par

\textsc{Michael Krivelevich} \par
\textsc{School of Mathematical Sciences, Tel Aviv University,
  Tel Aviv 6997801, Israel} \par
  \textit{Email:}
  \href{mailto:krivelev@tauex.tau.ac.il}{\texttt{krivelev@tauex.tau.ac.il}}\par
  Research supported in part by USA-Israel BSF grant 2018267 and by ISF grant 1261/17. \par
  \addvspace{\medskipamount}

\par

\textsc{Peleg Michaeli} \par
\textsc{School of Mathematical Sciences, Tel Aviv University,
  Tel Aviv 6997801, Israel} \par
  \textit{Email:}
  \href{mailto:peleg.michaeli@math.tau.ac.il}{\texttt{peleg.michaeli@math.tau.ac.il}} \par
  This research is supported by ERC starting grant 676970 RANDGEOM and by ISF grant 1207/15.\par
  \addvspace{\medskipamount}

}}

\usepackage{subcaption}
\usepackage{tikz}
\usetikzlibrary{calc}

\newcommand{\KC}{\mathsf{KC}}
\newcommand{\NP}{\mathbf{NP}}
\newcommand{\Small}{\textsc{Small}}

\usepackage{tabto}

\begin{document}
\maketitle

\begin{abstract}
  We present an explicit connected spanning structure that appears in a random graph just above the connectivity threshold with high probability.
\end{abstract}

\section{Introduction}
\label{sec:intro}
The \defn{binomial random graph} $G(n,p)$ is a graph on $n$ vertices, in which every pair of vertices is connected independently with probability $p$.
It is a well known and thoroughly studied model (see, e.g.,~\cites{Bol,FK,JLR}).
A fundamental result, due to Erd\H{o}s and R\'enyi~\cite{ER59}, is that $G(n,p)$ exhibits a sharp threshold for connectivity at $p=p(n)=\log{n}/n$\footnote{Here and later the logarithms have natural base.}.
More precisely, setting $p=(\log{n}+f(n))/n$ and $G\sim G(n,p)$, if $f(n)\to\infty$ then $G$ is with high probability\footnote{With probability tending to $1$ as $n$ tends to infinity.} (\whp{}) connected, and if $f(n)\to-\infty$ then $G$ is \whp{} not connected.
This threshold coincides with the threshold for the disappearance of isolated vertices (vertices of degree $0$), and, in fact, isolated vertices are the bottleneck for connectivity in a stronger sense.

Evidently, a connected graph has a spanning tree.
This raises the natural question of {\em which} spanning trees can we expect to find in a connected random graph.
The standard proof for the connectivity threshold is obtained using the dual definition of connectivity, namely, by showing that \whp{} there is an edge in every cut; this seems to provide no hint of which trees appear above the threshold.
One can check, however, that just above the connectivity threshold a random graphs contains $\Theta(\log{n})$ vertices of degree $1$.
Obviously, these vertices must all be leaves in any spanning tree.
In particular, one cannot expect to find a Hamilton path (or any spanning tree with a constant number of leaves).
In this work, we present an \emph{explicit} spanning tree that appears in a random graph just above the connectivity threshold \whp{}.
In fact, we prove something stronger, by presenting a concrete unicyclic connected spanning subgraph.
Let $n,t,\ell$ be integers such that $t\cdot (\ell+1) \le n$. A \defn{KeyChain} with parameters $n,t,\ell$, denoted $\KC(n,t,\ell)$, is a cycle on $n-t$ vertices with additional $t$ vertices of degree $1$ (``keys") which have distinct neighbours in that cycle, where the distance between two consecutive such neighbours is $\ell$.
Formally, it is the graph $H=(V,E)$ with $V=[n]$ and
\[
  E = \{ \{i,i+1\} \mid i \in [n-t-1] \}
  \cup \{ \{n-t,1\} \}
  \cup \{ \{i\ell ,n-t+i\} \mid i\in [t] \}.
\]
See \cref{fig:keychain} for an example.
Consider the following sequence: $a_1 = 1,\ a_{j+1} = \ceil{a_j \cdot \log n / 100}$, and set $j_0$ to be the minimum index $j$ for which $a_j \ge 10n / \log n$.
Let further $t:=\floor{\log{n}}$ and $\ell=2j_0$.

\begin{figure*}[t]
  \captionsetup{width=0.879\textwidth,font=small}
  \centering
  \begin{tikzpicture}
    \tikzset{vertex/.style={fill,circle,inner sep=1.5pt}}
    \foreach \i in {1,2,...,24} {
      \node[vertex] (\i) at (\i*15:4 and 1) {};
      \draw (\i) -- (\i*15+15:4 and 1);
    }
    \foreach \i in {1,2,...,5} {
      \node[vertex] (k\i) at ($(\i*45-60:4 and 1)+(\i*45-60:1)$) {};
      \draw (k\i) -- (\i*45-60:4 and 1);
    }
  \end{tikzpicture}
  \caption{$\KC(24,5,3)$.}
  \label{fig:keychain}
\end{figure*}
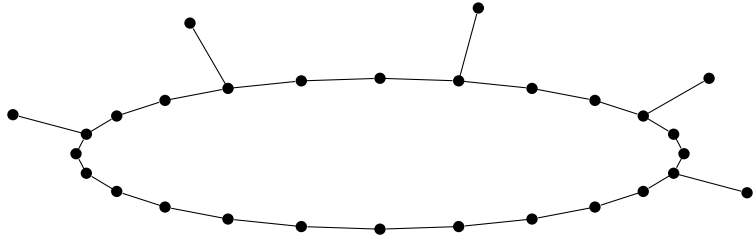

\begin{theorem}\label{thm:kc}
  Let $p$ be such that $np - \log{n} \rightarrow \infty$.
  Then, $G\sim G(n,p)$ \whp{} contains $\KC(n,t,\ell)$ as a subgraph.
\end{theorem}

As $\KC(n,t,\ell)$ is connected and spanning,
and in fact contains $\left\lceil \frac{n-t}{2} \right\rceil$ distinct (unlabelled) spanning trees when $n$ is sufficiently big,
\cref{thm:kc} gives a set of distinct trees, all of which appear in $G\sim G(n,p)$ \whp{}.
Our proof provides, however, further spanning (but not necessarily connected) graphs that appear \whp{} in random graphs above the connectivity threshold (such as KeyChains on linearly many vertices alongside a disjoint cycle spanning the remaining set of vertices).

\subsection{Discussion}
\paragraph{Hamiltonicity}
Koml\'os and Szemer\'edi~\cite{KS83} and independently Bollob\'as~\cite{Bol84} showed that the threshold for the appearance of a Hamilton cycle in random graphs is $p=(\log{n}+\log\log{n})/n$.
This coincides with the threshold for the disappearance of vertices of degree $1$, an obvious obstacle in obtaining a Hamilton cycle.
In particular, if $np-\log{n}-\log{\log{n}}\to\infty$ and $G\sim G(n,p)$, then $G$ contains a Hamilton path as a spanning tree.
Thus, our result is interesting, and perhaps surprising, for values of $p$ which satisfy $np=\log{n}+f(n)$, where $f(n)\to\infty$ but $f(n)-\log{\log{n}}\not\to\infty$.
Our result states that in this intermediate regime, we can still construct a \emph{specific} tree which contains a path that spans almost all of the vertices of the graph.

\paragraph{Spanning trees in random graphs}
\Cref{thm:kc} gives an explicit set of trees, each of which spans a binomial random graph just above the connectivity threshold \whp{}.
Which other spanning trees can we expect to find around the same edge density?
Recently Montgomery~\cite{Mon19} solved a conjecture by Kahn (see~\cite{KLW16}) according to which every $n$-vertex tree of maximum degree at most $\Delta$ appears in $G(n,p)$ \whp{} if $np\ge C\log{n}$ for $C=C(\Delta)$.
Earlier, Hefetz, Krivelevich and Szab\'o~\cite{HKS12} showed that every bounded degree $n$-vertex tree with either linearly many leaves or a linearly long bare path appears in $G(n,p)$ \whp{} if $np\ge (1+\eps)\log{n}$ for some $\eps>0$.

We wish to stress that our declared goal in this paper is to present one concrete family of spanning trees appearing typically at the threshold for connectivity, and this goal is achieved in \cref{thm:kc}. Undoubtedly a wider class of spanning trees can be shown to appear \textbf{whp} for this value of $p(n)$ using similar (but possibly more involved) techniques and arguments, but we decided not to pursue this goal here, aiming rather for relative simplicity.
The question of which trees are likely to appear in $G(n,p)$ if one only requires $np-\log{n}\to\infty$ therefore remains open.

Another possible extension of our main result would be to consider the random graph process, and to try and prove that some concrete spanning subgraph (say, our KeyChain) appears typically at the very moment the graph becomes connected. This would imply our result by the standard connections between random graph processes and models $G(n,p), G(n,m)$. Again we chose not to push in this direction, preferring simplicity.

\paragraph{The maximum common subgraph problem}
Finding a maximum common subgraph of two graphs
(sometimes called \emph{maximum common edge subgraph}, or MCES)
is an $\NP$-hard problem
(for example, there is a trivial polynomial reduction from the Hamiltonicity problem to MCES)
with real-world applications in computer science and in chemistry (see, e.g.,~\cites{Bok81,RGW02}).
\Cref{thm:kc} can be thought of as an explicit (connected) subgraph
which is typically common to random graphs past the connectivity threshold.
In particular, writing $M(G_1,G_2)$ for the size of a maximum common subgraph of $G_1,G_2$, \cref{thm:kc} gives $M(G_1,G_2)\ge n$ \whp{} for independently sampled $G_1,G_2\sim G(n,p)$, where $np-\log{n}\to\infty$.
The following proposition, whose proof appears in \cref{sec:mcs}, shows that this bound is asymptotically tight.
\begin{proposition}\label{prop:mcs}
  For every $\eps>0$ there exists $\delta>0$ for which the following holds.
  Suppose $p\le n^{-1+\delta}$, and let $G_1,G_2\sim G(n,p)$ be two independent random graphs.
  Then, \whp{}, $M(G_1,G_2)\le(1+\eps)n$.
\end{proposition}

\subsection{Proof outline}
We begin by listing useful properties of random graphs just above the connectivity threshold.
With these properties in hand, we continue as follows.
First, we take care of all ``keys'' of the KeyChain:
these consist of all vertices of degree $1$ in the graph, in addition to some vertices of degree $2$.
We aim to connect neighbours of these keys by equal-length paths, constructing a comb-like graph.
Eventually, we would want to connect the ends of the comb by a path which spans the remaining vertices.
This cannot be done naively, however, since some of the remaining vertices may have most (or all) of their neighbours inside the comb.
To overcome this difficulty we make a preparatory step in which we put aside small degree vertices with their neighbours.
This is stated precisely in \cref{sec:const}.
The ``comb'' is then found in the set of the remaining vertices.

\paragraph{Organisation of the paper}
We start by reviewing some preliminaries in \cref{sec:prem}.
In \cref{sec:random} we list and prove useful (and mostly standard) properties of random graphs just above the connectivity threshold.
In \cref{sec:const} we prepare our graph, handling future ``keys'' and small degree vertices, construct the ``comb'' and close it to a KeyChain, concluding the proof of \cref{thm:kc}.
We finish by a quick proof of \cref{prop:mcs} in \cref{sec:mcs}.

\section{Preliminaries}
\label{sec:prem}

In this section we provide several definitions and results to be used in this paper.

\subsection{Notation} \label{subsec:prem:notation}

The following graph theoretic notation is used.
For a graph $G=(V,E)$ and two disjoint vertex subsets $U,W\subseteq V$, we let $E_G(U,W)$ denote the set of edges of $G$ adjacent to exactly one vertex from $U$ and one vertex from $V$, and let $e_G(U,W)=|E_G(U,W)|$.
Similarly, $E_G(U)$ denotes the set of edges spanned by a subset $U$ of $V$, and $e_G(U)$ stands for $|E_G(U)|$.
The (external) neighbourhood of a vertex subset $U$, denoted by $N_G(U)$, is the set of vertices in $V\sm U$ adjacent to a vertex of $U$, and for a vertex $v\in V$ we set $N_G(v)=N_G(\{v\})$.
The degree of a vertex $v\in V$, denoted by $d_G(v)$, is its number of incident edges.
For an integer $0 \leq i < |V|$, we let $D_i = D_i(G)$ be the set of vertices of degree $i$ in $G$,
and let $D_{\le i}=D_{\le i}(G):=\bigcup_{j=0}^i D_j$.
Finally, we let $n_i = n_i(G) := |D_i|$ and $n_{\le i} = n_{\le i}(G) := |D_{\le i}|$.
In the above notation, we sometimes omit the subscript $G$ if the graph $G$ is clear from the context.
We occasionally suppress the rounding notation to simplify the presentation. For two functions $f=f(n)$, $g=g(n)$ we write $f\sim g$ to indicate that $f = (1+o(1))g$.

\subsection{Probabilistic bounds} \label{subsec:prem:prob}

We will make use of the following useful bound (see, e.g., in,~\cite{JLR}*{Chapter 2}).

\begin{theorem}[Chernoff bounds]\label{thm:chernoff}
  Let $X=\sum_{i=1}^n X_i$, where $X_i\sim\Bernoulli(p_i)$ are independent, and let $\mu=\E{X}=\sum_{i=1}^n p_i$.
  Let $0<\alpha<1<\beta$.
  Then
  \begin{align*}
    \pr(X\le \alpha\mu) &\le \exp(-\mu(\alpha\log{\alpha}-\alpha+1)),\\
    \pr(X\ge \beta\mu) &\le \exp(-\mu(\beta\log{\beta}-\beta+1)).
  \end{align*}
\end{theorem}

The following are trivial yet useful bounds.
\begin{claim}\label{cl:bin:uptail}
  Let $X\sim\Bin(n,p)$ with $\mu=np$ and let $1\le k\le n$.
  Then
  \begin{equation*}
    \pr(X\ge k) \le \left(\frac{enp}{k}\right)^k.
  \end{equation*}
\end{claim}

For a proof see, e.g.,~\cite{FKMP18}.

\begin{claim}\label{cl:bin:lowtail}
  Let $X\sim\Bin(n,p)$ with $\mu=np$, write $q=1-p$ and let $1\le k\le np/q$.
  Then
  \begin{equation*}
    \pr(X\le k) \le \left(\frac{enp}{kq}\right)^k e^{-np}.
  \end{equation*}
\end{claim}

\subsection{P\'{o}sa's lemma and corollaries} \label{subsec:prem:posa}

For an overview of the rotation--extension technique, we refer the reader to~\cites{Kri16}.

\begin{lemma}[P\'{o}sa's lemma~\cite{Pos76}]\label{lem:Posa}
Let $G$ be a graph, let $P = v_0,\dots,v_t$ be a longest path in $G$, and let $R$ be the set of all $v \in V(P)$ such that there exists a path $P'$ in $G$ with $V(P') = V(P)$ and with endpoints $v_0$ and $v$.
Then $|N(R)| \leq 2|R|-1$.
\end{lemma}

Recall that a non-edge of $G$ is called a \defn{booster} if adding it to $G$ creates a graph which is either Hamiltonian or whose longest path is longer than that of $G$.
For a positive integer $k$ and a positive real $\alpha$ we say that a graph $G=(V,E)$ is a \defn{$(k,\alpha)$-expander} if $|N(U)|\ge\alpha|U|$ for every set $U\subseteq V$ of at most $k$ vertices.
The following is a widely-used fact stating that $(k,2)$-expanders have many boosters. For a proof see, e.g.,~\cite{Kri16}.

\begin{lemma}\label{lem:posa:boosters}
  Let $G$ be a connected $(k,2)$-expander which contains no Hamilton cycle.
  Then $G$ has at least $(k+1)^2/2$ boosters.
\end{lemma}
It is not hard to see that an $(n/4,2)$-expander on $n$ vertices is connected, a fact that we will use later on.
We will also use the following deterministic lemma from \cite{GKM21}.

\begin{lemma}[\cite{GKM21}*{Lemma 4.6}]\label{lem:expander}
Let $m,d \geq 1$ be integers, and let $H$ be a graph on $h\ge 4m$ vertices satisfying the following properties:
\begin{enumerate}
    \item $\delta(H) \geq 2$;
    \item No vertex $v \in V(H)$ with $d(v) < d$ is contained in a $3$- or a $4$-cycle, and every two distinct vertices $u,v \in V(H)$ with $d(u),d(v) < d$ are at distance at least $5$ apart;
    \item Every set $F \subseteq V(H)$ of size at most $5 m$ spans at most $d|F| / 10$ edges;
    \item There is an edge between every pair of disjoint sets $F_1,F_2 \subseteq V(H)$ of size $m$ each.
\end{enumerate}
Then $H$ is a $(h/4,2)$-expander (and, in particular, it is connected).
\end{lemma}

\section{Properties of random graphs}
\label{sec:random}
We now present and prove some typical properties of random graphs above the connectivity threshold, to be used in the proof of \cref{thm:kc}.
The properties we present are fairly standard and the consequent proofs are mainly technical and may get tedious, so a first time reader may wish to skip them.

Denote
\[
  \Small = \Small(G) := D_{\le\log{n}/10}(G).
\]
Additionally, set
\[\gamma=10^{-4}.\]

\begin{lemma}\label{lem:gnp:prop}
  Let $1\ll f(n)\ll \log\log{n}$, $p=(\log{n}+f(n))/n$ and $G\sim G(n,p)$.
  Then, \whp{}, $G$ has the following properties:
  \begin{description}[leftmargin=!,labelwidth=\widthof{\bfseries (P1)}]

    \item[\namedlabel{P:maxdeg}{(P1)}]
      $\Delta(G) \leq 10\log n$;

    \item[\namedlabel{P:numkeys}{(P2)}]
      $|D_1|\le\log{n}\le|D_2|$;

    \item[\namedlabel{P:smalldist}{(P3)}]
      There is no path of length at most $0.2\log n / \log \log n$ in $G$ whose (possibly identical) endpoints lie in $\Small$;

    \item[\namedlabel{P:smallsmall}{(P4)}]
      $|\Small\cup N(\Small)|\le n^{0.6}$;

    \item[\namedlabel{P:smallsets}{(P5)}]
    There in no $U\subseteq V(G)$ with $|U|\leq 10n/\log n$, $e(U,V(G)\sm U) \geq |U|\log n/11 $ such that $|N(U)| \leq |U|\log n/18$;

    \item[\namedlabel{P:local:sparse}{(P6)}]
      Every set $U \subseteq V(G)$ of size at most $\gamma n/5000$ spans at most $\gamma\log n\cdot|U|/1000$ edges;

    \item[\namedlabel{P:intersect}{(P7)}]
      For every $U,W\subseteq V(G)$ disjoint with $10 n/\log n\leq |U|,|W|\leq n/9$, $|N(U)\cap N(W)|\geq n/9$;

    \item[\namedlabel{P:pseudorandom}{(P8)}]
      For every $U,W\subseteq V(G)$ disjoint with $|U|,|W|\geq \frac{\gamma n}{25000}$, $|E(U,W)|\ge \frac{1}{2}|U||W|\log{n}/n$;

  \end{description}
\end{lemma}

\begin{proof}[Proof of \ref{P:maxdeg}]
  Since $d(v)\sim\Bin(n-1,p)$ for every $v\in V(G)$, we have
  \[
    \pr(d(v)\ge 10\log{n})
    \le \binom{n}{10\log{n}}p^{10\log{n}}
    \le \left(\frac{enp}{10\log{n}}\right)^{10\log{n}} = o(1/n),
  \]
  and the statement follows by the union bound.
\end{proof}

\begin{proof}[Proof of \ref{P:numkeys}]
  The probability that a given vertex is of degree $1$ is at most $np(1-p)^{n-2}=o(\log{n}/n)$.
  Thus $\E[n_1]=o(\log{n})$ and by Markov's inequality $\pr(n_1>\log{n})=o(1)$.
 	The expectation of $n_2$ is $n\cdot \binom{n-1}{2} p^2 (1-p)^{n-3}=\omega(\log{n})$.
   On the other hand, the variance is
 	\begin{equation*}
  \begin{aligned}
 	  \Var \left[ n_2 \right]
    & = \E\left[ {n_2} ^2\right] -  \E\left[ {n_2}\right] ^2 \\
 	  & = -\E\left[ {n_2}\right] ^2 + \E\left[ {n_2}\right]
        + \sum_{\substack{(u,v)\in V(G)^2\\u\ne v}} \pr\left( d_G(u) = d_G(v) = 2 \right) \\
    & =   (-1 + o(1))\cdot \E\left[ {n_2}\right] ^2 + n(n-1)\cdot \left( \binom{n-2}{2}^2p^4(1-p)^{2n-7} + np^3(1-p)^{2n-6} \right) \\
    & =  o\left( \E\left[ {n_2}\right] ^2 \right) ,
  \end{aligned}
 	\end{equation*}
 	which, by Chebyshev's inequality, implies that $\pr(n_2 < \log n) = o(1)$.
\end{proof}

\begin{proof}[Proof of \ref{P:smalldist}]
  Write $L:=0.2\log{n}/\log{\log{n}}$.
  Let $1\le\ell\le L$ and let $P=(v_0,\ldots,v_\ell)$ be a sequence of $\ell+1$ distinct vertices from $V(G)$, with the one possible exception $v_0=v_\ell$.

  Suppose first that $v_0\ne v_\ell$.
  Let $S:= V(G) \sm \{v_0,v_1,v_{\ell-1},v_\ell\}$,
  let $\cA_P$ be the event that $P$ is contained in $G$ as a path, and
  let $\cA_d$ be the event that $d(\{v_0,v_\ell\},S) \le \log{n}/5$.
  By \cref{thm:chernoff} with $\alpha = \frac{\log{n}}{5\cdot (2n-5)p} = 1/10 + o(1)$, we obtain that
  \[
    \pr(\cA_d)
      \le \exp \left( -(2n-5)p \cdot ( \alpha \log{\alpha} - \alpha +1 ) \right)
      \le n^{-1.3}.
  \]
  The events $\cA_P,\cA_d$ are independent, hence $\pr(\cA_P\land \cA_d) \le p^\ell n^{-1.3}$.
  Let $\cA$ be the event that there exists a path $P=v_0,\ldots,v_\ell$ with $1\le\ell\le L$ in $G$ such that $\cA_P$ and $d(v_0),d(v_\ell)\le\log{n}/10$, the probability of which is at most $\pr(\cA_P\land \cA_d)$.
  By the union bound, summing over all sequence lengths and all sequences, we get
  \[
    \pr(\cA)
    \le \sum_{\ell=1}^L n^{\ell+1}p^\ell n^{-1.3}
    \le \sum_{\ell=1}^L n^{\frac{\log{(np)}}{\log{n}}\cdot \ell - 0.3}
    \le L\cdot n^{\frac{1.1\log{\log{n}}}{\log{n}}\cdot L -0.3}
    \le L\cdot n^{-1/20} = o(1).
  \]

  The case $v_0=v_\ell$ is similar.
  Let $S:=V(G)\sm \{v_1,v_{\ell-1}\}$ and let $\cA_d$ be the event $d(v_0,S)\le\log{n}/10$.
  Once again by \cref{thm:chernoff}, $\pr(\cA_d)\le n^{-0.6}$, and the events $\cA_P,\cA_d$ are independent, hence $\pr(\cA_p\land \cA_d)\le p^\ell n^{-0.6}$.
  Let $\cA'$ be the event that there exists a cycle $P$ of length $3\le\ell\le L$ such that $\cA_P$ and $d(v_0)\le\log{n}/10$.
  By the union bound, $\pr(\cA')\le \sum_{\ell=3}^L n^\ell p^\ell n^{-0.6}=o(1)$.

  Finally, observe that $\pr(G\in \ref{P:smalldist}) \ge 1-\pr (\cA) - \pr (\cA ')$, and so the statement is obtained.
\end{proof}

\begin{proof}[Proof of \ref{P:smallsmall}]
  Thus, by \cref{thm:chernoff} we have $\pr(d(v)\le\log{n}/10)\le n^{-0.6}$, and therefore by Markov's inequality $|\Small|\le n^{0.5}$ \whp{}.
  By the definition of $\Small$ we have that $|\Small\cup N(\Small)|\le n^{0.6}$ \whp{}.
\end{proof}

\begin{proof}[Proof of \ref{P:smallsets}]
  For $U\subseteq V(G)$ with $|U|=k$, let $\cA_U$ be the event that $e(U,V(G)\sm U)\ge k\log{n}/11$ and $|N(U)| \le k\log{n}/18$.
  On this event, there exists $W\subseteq V(G)$ of size $k\log{n}/18$, disjoint of $U$, which contains $N(U)$.
  Denote this event by $\cA_{U,W}$.
  Evidently, using \cref{cl:bin:uptail},
  \[
  \begin{aligned}
    \pr(\cA_{U,W})
    &\le \pr\left( \Bin(k^2\log{n}/18,p) \ge k\log{n}/11 \right)
     \cdot (1-p)^{k(n-k-k\log{n}/18)}\\
    &\le\left( \frac{11ek p}{18} \right)^{k\log{n}/11}
     \cdot e^{-pk(n-k-k\log{n}/18)}
     \le \left( \frac{11ek p}{18} \right)^{k\log{n}/11}
          \cdot e^{-0.4npk}.
  \end{aligned}
  \]
  For $k\le 10n/\log{n}$ let $\cA_k$ denote the event that there exists $U$ with $|U|=k$ for which $\cA_U$ occurs.
  By the union bound over all possible choices of $U,W$ we have that
  \[
  \begin{aligned}
    \pr(\cA_k) &\le \binom{n}{k+k\log{n}/18} \binom{k+k\log{n}/18}{k}
    \cdot\pr(\cA_{U,W})\\
    &\le \left(\frac{en}{k+k\log{n}/18}\right)^{(1/18+o(1))k\log{n}}
     \cdot \log^k{n}
    \cdot\pr(\cA_{U,W})\\
    &\le \left(\frac{50n}{k\log{n}}\right)^{(1/18+o(1))k\log{n}}
    \cdot
    \left( \frac{11ek p}{18} \right)^{k\log{n}/11}
              \cdot e^{-0.4npk}\\
    &\le \left[\left(\frac{n}{k\log{n}}\right)^{1/18-1/11+o(1)}
        \cdot 50^{1/17}
        \cdot 2^{1/11}
                  \cdot e^{-0.4}\right]^{k\log{n}}\\
    &\le \exp\left( k\log{n}\cdot \left(
        \log(10)/25 + \log(50)/17 + \log(2)/11 - 0.4
    \right) \right) = e^{-\Omega(k\log{n})}.
  \end{aligned}
  \]
  Finally, let $\cA$ denote the event that there exists $U$ with $1\le |U|\le 10n/\log{n}$ for which $\cA_U$ occurs.
  By the union bound over all possible cardinalities $k=|U|$ we have
  \[
    \pr(\cA)
      \le \sum_{k=1}^{10n/\log{n}}\pr(\cA_k)
      \le \sum_{k=1}^{\infty} e^{-\Omega(k\log{n})} = o(1).\qedhere
  \]

\end{proof}

\begin{proof}[Proof of \ref{P:local:sparse}]
  Fix $U\subseteq V(G)$ and set $\gamma'=\gamma/5000$.
  By \cref{cl:bin:uptail}, the probability that $|E(U)|\ge m$ for some $m\ge 1$ is at most $(e|U|^2p/m)^m$.
  Hence, by the union bound, the probability that there is a subset $U\subseteq V(G)$ negating \ref{P:local:sparse} is at most
  \begin{equation*}
  \begin{aligned}
    &\phantom{\le{}}  \sum_{k=1}^{\gamma' n}
      \binom{n}{k}\left(\frac{ekp}{5\gamma'\log{n}}\right)^{5\gamma'\log{n}\cdot k}\\
    &\le \sum_{k=1}^{\gamma' n}
      \left(\left(\frac{en}{k}\right)
        \cdot \left(\frac{0.6k}{\gamma' n}\right)^{5\gamma'\log{n}}\right)^k\\
    &\le \sum_{k=1}^{\gamma' n}
      \left(\frac{e}{\gamma'} \cdot 0.6^{\Omega(\log{n})}\right)^k
     = \sum_{k=1}^{\gamma' n} o(1)^k = o(1),
  \end{aligned}
  \end{equation*}
  and the statement follows.
\end{proof}

\begin{proof}[Proof of \ref{P:intersect}]
  If $U,W\subseteq V(G)$ are disjoint with $10 n/\log n\leq |U|,|W|\leq n/9$, and $|N(U)\cap N(W)| < n/9$, then there is a set $X\subseteq V(G)\setminus (U\cup W)$ of size $0.6n$ such that every vertex $x\in X$ is connected to no vertex in $U$, or connected to no vertex in $W$. By the union bound, the probability of this is at most

	\begin{equation*}
  \begin{aligned}
	  &\sum_{k=\frac{10n}{\log n}}^{n/9}
      \sum_{s=\frac{10n}{\log n}}^{n/9}
      \binom{n}{k} \binom{n}{s} \binom{n}{0.6n}
      \cdot \left( (1-p)^k + (1-p)^s \right) ^{0.6n} \\
	  &\leq n^2\cdot \left( 9e \right)^{2n/9}
      \cdot (2e)^{0.6n}
      \cdot 2^{0.6 n}
      \cdot \exp \left( -\frac{10n}{\log n}p\cdot 0.6n \right) \\
	  &\leq \exp \left(
        \left( o(1) + \log (9e)/2 + 2\log 2 + 1 - 10 \right) \cdot 0.6 n
      \right)
	  = o(1),
  \end{aligned}
	\end{equation*}
  and the statement follows.
\end{proof}

\begin{proof}[Proof of \ref{P:pseudorandom}]
If $U,W\subseteq V(G)$ are disjoint, and $|U|,|W| \geq \gamma n/25000$, then by \cref{thm:chernoff}, the probability that $e(U,W) < \frac{1}{2}|U||W|\log n/n \leq \frac{2}{3}\E (e(U,W))$ is of order
\[
  \exp \left( -\Omega(\E (e(U,W)))\right)
    = \exp \left(-\Omega (|U||W|\log n /n) \right)
    = \exp \left(-\Omega (n\log n)\right) .
\]
By the union bound, the probability that such $U,W$ exist is at most
\[
  3^n\cdot  \exp (-\Omega (n\log n)) = o(1).\qedhere
\]
\end{proof}

The following lemma describes yet another property of $G(n,p)$, which might need further explanation.
Recall that $G\sim G(n,p)$ is not assumed to be Hamiltonian, and in the relevant regime typically does not contain a Hamilton path.
We will need, however, to find a path (and, in fact, many such paths) which spans a large predetermined portion of its vertices.
It is not hard to see that subgraphs spanned by carefully chosen (large) sets of vertices are (very) good expanders.
Our way to argue that they are Hamiltonian, however, will be to show that they contain {\em sparser} expanders.
To this end, we will use the method of ``random sparsification'' which has become a fairly standard tool in the study of Hamiltonicity (see, e.g., in~\cites{BKS11mb,BKS11p,Kri16,FKMP18}).
The main idea behind this, which is the essence of the next lemma, is that we can show that \whp{} {\em every} sparse expander will have a relative booster.

\begin{lemma}\label{lem:boosters}
  Let $1\ll f(n)\ll \log\log{n}$, $p=(\log{n}+f(n))/n$ and $G\sim G(n,p)$.
  Then, \whp{},
  for every $W \subseteq V(G)$ of size $|W|=h \geq n/10$ and for every $(h/4,2)$-expander $H$ on $W$ which is a subgraph of $G$ with at most $\gamma n \log n/100$ edges, $G$ contains a booster with respect to $H$.
\end{lemma}

\noindent For later references we will name the property described above
{\bf \namedlabel{P:boosters}{(P9)}}.

\begin{proof}
  Fix $W\subseteq V(G)$ such that $h:=|W|\ge n/10$, and let $H$ be an $(h/4,2)$-expander on $W$ with $m$ edges, that is a subgraph of $G$.
  Recall that $H$ is connected.
  Hence, by \cref{lem:posa:boosters} we know that $G[W]$ has at least $n^2/3200$ boosters with respect to $H$.
  Thus, the probability that $G[W]$ contains $H$ but no booster thereof is at most
  \[
    p^m \cdot (1-p)^{n^2/3200}
      \le p^m \cdot \left(1-\frac{\log{n}}{n}\right)^{n^2/3200}
      \le \left(\frac{2\log{n}}{n}\right)^m
        \cdot \exp\left(-n\log{n}/3200\right).
  \]
  Write $\beta=\gamma/100$.
  As there are at most $2^n$ choices for $W$ and at most $\binom{n^2}{m}\le\left(en^2/m\right)^m$ choices for $H$ for each $1\le m\le \beta n\log{n}$, we have, by the union bound, that the probability that there exist such $W$ and $H$ for which $G[W]$ does not contain a booster with respect to $H$ is at most
  \begin{equation}\label{eq:lem:boosters}
    2^n \cdot \exp\left(-n\log{n}/3200\right)
      \cdot \sum_{m=1}^{\beta n\log{n}} \left(\frac{2en\log{n}}{m}\right)^m.
  \end{equation}
  Set $g(m)=(2en\log{n}/m)^m$ and observe that $g'(m)=g(m)\cdot(\log(2en\log{n}/m)-1)$, which is positive for $1\le m\le \beta n\log{n}$.
  Thus, the sum in \eqref{eq:lem:boosters} can be bounded from above by
  \[
    \beta n\log{n}\cdot \left(2e/\beta\right)^{\beta n\log{n}}
    = \exp\left( (\beta\log(2e/\beta) + o(1))n\log{n} \right),
  \]
  which, recalling that $\beta=\gamma/100=10^{-6}$, is smaller than $\exp(n\log{n}/3201)$,
   and thus \eqref{eq:lem:boosters} tends to $0$ as $n\to\infty$.
\end{proof}

\section{Constructing a KeyChain}
\label{sec:const}
In \cref{sec:random} we have identified useful properties which are satisfied by random graphs \whp{}.
In this section we will assume our graph possesses these properties, and show that this deterministically implies that the graph contains a KeyChain.

For convenience, let us repeat some definitions from \cref{sec:intro,sec:random}.
Consider the following sequence: $a_1 = 1,\ a_{j+1} = \ceil{a_j \cdot \log n / 100}$, and set $j_0$ to be the minimum index $j$ for which $a_j \ge 10n / \log n$.
Let further
\[
t:=\floor{\log{n}}\quad\text{and}\quad \ell=2j_0.
\]
Finally, recall that $\gamma=10^{-4}$.
In this section we prove the following lemma, which, when put together with \cref{lem:gnp:prop,lem:boosters}, completes the proof of \cref{thm:kc}.

\begin{lemma}\label{lem:kc}
  Any $n$-vertex graph $G$ satisfying Properties \ref{P:maxdeg}--\ref{P:boosters}
  contains $\KC(n,t,\ell)$ as a subgraph.
\end{lemma}

Our plan is as follows.
We want to construct a ``comb'', which consists of $t$ keys (all vertices of degree $1$ in addition to some vertices of degree $2$), and equal-length paths between neighbours of consecutive keys.
We will then want to connect the endpoints of the comb by a path which spans the remaining set of vertices.
As hinted in the introduction, we cannot carelessly do so, as some vertices outside the comb might have most or all of their neighbours inside the comb.
Instead, we have to make a preparatory step, in which we put aside vertices of small degree (except the future ``keys'' of the KeyChain) along with their neighbourhoods.
This preparatory step is \cref{lem:prep}.
Given the partition in \cref{lem:prep}, we construct a comb in its large part (in \cref{lem:paths}), and then connect the endpoints with a path that spans the remaining set of vertices.
This is depicted in \cref{fig:hamkc}.

~

In the next lemmas we assume $G$ is an $n$-vertex graph satisfying Properties \ref{P:maxdeg}--\ref{P:boosters}.

\begin{lemma}\label{lem:prep}
  There exist a partition $V(G)=V^\star\cup V'$ with $|V^\star|\sim 2\gamma n$ and a set $K\subseteq V'$ with $|K|=t$ for which the following holds:
  \begin{description}
  \item[(a)] $D_1(G)\subseteq K\subseteq D_{\le 2}(G)$;
  \item[(b)] $N(K)\subseteq V'$;
  \item[(c)] $d(v,V^\star)\le 200\gamma\log{n}$ and $d(v,V')\ge \log{n}/20$ for every $v\in V'\sm K$;
  \item[(d)] If $K\subseteq X\subseteq V'$ satisfies $|X|\le n/2$ and $w_1,w_2\in X\sm K$ then there exist $z_1\sim w_1$ and $z_2\sim w_2$ in $V(G)\sm X$, and a Hamilton path from $z_1$ to $z_2$ in $G[V(G)\sm X]$.
  \end{description}
\end{lemma}

Here, $K$ will be the set of keys for our future construction, and $V'$ will host the comb, with conditions {\bf (b)},{\bf (c)} ensuring that its construction is indeed possible. Finally, condition {\bf (d)} ensures that the comb can be extended into a copy of $\KC(n,t,\ell)$ by plugging it as $X$ and the two endpoints of the comb's path as $w_1,w_2$.

The proof of \cref{lem:prep} is based on two ingredients.

The first ingredient (\cref{lem:partition}) takes care of the actual partition promised by \cref{lem:prep}. In this step we take measures to ensure that our partition satisfies the desired conditions. In particular, vertices of $\Small \setminus K$, and their neighbours, are placed in $V^\star$. This serves a dual purpose: we ensure that the minimum degree of the graph spanned by $V'$ is at least logarithmic, thus aiding us with the construction of the comb, and that the minimum degree after removing the comb is at least 2, which is a necessary condition for the completion of the comb into $\KC(n,t,\ell)$.

The second ingredient (\cref{lem:ham}), which is the core of the proof, gives {\bf (d)}, by showing that inside these ``well-prepared'' sets, one can find Hamilton paths with linearly many distinct endpoints, emerging from a given vertex.

\begin{lemma}\label{lem:partition}
  There exist disjoint sets $K,U_1,U_2\subseteq V(G)$ with $|K|=t$ and $|U_1|,|U_2|\sim\gamma n$ for which the following holds.
  Write $V'=V(G)\sm(U_1\cup U_2)$.
  Then
  \begin{description}
      \item[(a)] $D_1(G)\subseteq K\subseteq D_{\le 2}(G)$;
      \item[(b)] For every $v\in K$, $N(v)\subseteq V'\sm K$;
      \item[(c)] If $v\notin\Small$ then $\gamma\log{n}/100 \le d(v,U_1),d(v,U_2)\le 100\gamma \log n$ and $d(v,V')\ge \log{n}/20$;
      \item[(d)] If $v\in\Small\sm K$ then $v$ and all of its neighbours are in $U_1$.
  \end{description}
\end{lemma}

\begin{proof}
  The proof involves an application of the symmetric form of the Local Lemma (see, e.g.,~\cite{AS}*{Chapter~5}; a similar application appears in~\cite{HKS12} and in~\cite{GKM21}).
  Write $V=V(G)$, $\alpha=1/10$, $s=1/\gamma$. Let $r=\floor{n/s}\sim\gamma n$ and let $A_1,\ldots,A_r,Z$ be a partitioning of the vertices of $G$ into $r$ ``blobs'' $A_i$ of size $s$ and an extra set $Z$ with $0\le |Z|< s$.
  For $j\in[r]$ let $(x_j^1,x_j^2)$ be a uniformly chosen pair of distinct vertices from $A_j$.
  For $i=1,2$ define $U_i'=\{x_j^i\}_{j=1}^r$.
  Clearly, $|U_1'|=|U_2'|=r$ and $U_1'\cap U_2'=\es$.
  For every $v\notin\Small$ let $\cB_v^-$ be the event that $d(v,U_i') < 2\gamma\alpha^2\log{n}$ for some $i=1,2$, and let $\cB_v^+$ be the event that $d(v,U_i') > \gamma\alpha^{-2}\log{n}/2$ for some $i=1,2$.
  For such $v$, let $L(v)$ be the set of blobs that contain neighbours of $v$,
  namely, $L(v)=\{A_j:\ N(v)\cap A_j\ne\es\}$.
  For $j\in[r]$ write $n_j(v)=|N(v)\cap A_j|$,
  and note that $\sum_j n_j(v) \ge d(v)-s \ge \alpha\log{n}/2$.
  For $i=1,2$ and $j\in[r]$ let $\chi_j^i(v)$ be the indicator of the event that $x_j^i$ is a neighbour of $v$, and note that $\E\chi_j^i(v)=\gamma n_j(v)$.
  Observe that for $i=1,2$, $d(v,U_i') = \sum_j \chi_j^i(v)$, hence $\E[d(v,U_i')]=\gamma\sum_j n_j(v)\ge\gamma\alpha\log{n}/2$.
  Thus, by \cref{thm:chernoff}, $\pr(\cB_v^-)\le n^{-c_-}$ for some $c_->0$.
  Similarly, by \ref{P:maxdeg}, $\E[d(v,U_i')]=\gamma\sum_j n_j(v)\le \gamma\alpha^{-1}\log{n}$.
  Thus, by \cref{thm:chernoff}, $\pr(\cB_v^+)\le n^{-c_+}$ for some $c_+>0$.
  We conclude that for $\cB_v=\cB_v^-\cup\cB_v^+$ we have $\pr(\cB_v)\le n^{-c}$ for some $c>0$.

  For two distinct vertices $u,v\notin\Small$ say that $u,v$ are \defn{related} if $L(u)\cap L(v)\ne\es$.
  For a vertex $u\notin\Small$ let $R(u)$ be the set of vertices in $V\sm\Small$ which are related to $u$, and note that $|R(u)|\le s\Delta(G)^2$, which is, by \ref{P:maxdeg}, at most $C\log^2{n}$ for some $C>0$.
  Note that $\cB_u$ is mutually independent of the set of events $\{\cB_v\mid v\in (V\sm\Small)\sm R(u)\}$.
  We now apply the symmetric case of the Local Lemma: observing that $e n^{-c} \cdot C\log^2{n}<1$ (for large enough $n$), we get that with positive probability, none of the events $\cB_v$ occur.
  We choose $U_1',U_2'$ to satisfy this.

  Choose a set $K$ of size $t$ arbitrarily to satisfy \textbf{(a)}; this is possible due to \ref{P:numkeys}.
  Write $K^+=K\cup N(K)$ and $S^+=\Small\cup N(\Small)\supseteq K^+$.
  Note that $|K^+|\le 3\log{n}$ by~\ref{P:numkeys} and that $|S^+|\le n^{0.6}$ by~\ref{P:smallsmall}.
  Define $U_1=(U_1'\cup S^+)\sm K^+$ and $U_2=U_2'\sm S^+$ and observe that $|U_1|,|U_2|\sim \gamma n$.
  Due to \ref{P:smalldist}, $N(\Small\sm K)\subseteq S^+\sm K^+$, hence
  the construction satisfies \textbf{(d)}.
  Let $v\notin\Small$.
  The fact that $G$ satisfies \ref{P:smalldist} implies that $v$ has at most $1$ neighbour in $S^+$.
  Thus, for every $v\notin\Small$ and $i=1,2$ it holds that $|d(v,U_i)-d(v,U_i')|\le 1$, and, in addition, $d(v,V')\ge d(v,V\sm(U_1'\cup U_2'))-1 \ge \alpha\log{n}/2$, hence \textbf{(c)} is satisfied.
  By the discussion above and by \ref{P:smalldist}, \textbf{(b)} is also satisfied.
  We have thus proved the statement.
\end{proof}

\begin{lemma}\label{lem:ham}
  If $W\subseteq V(G)$ is a vertex subset such that $|W| \geq n/10$, $\delta(G[W])\ge 2$, and such that for every $v\in W$ we have $d(v,W)\ge \min\{ d(v), \gamma \log n/100\}$,
  then for every $w \in W$ there exists $R\subseteq W$ with $|R|\ge n/40$ such that for each $y\in R$, there is a Hamilton path in $G[W]$ whose endpoints are $w$ and $y$.
\end{lemma}

\begin{proof}
  Write $\beta=\gamma/100$ and set $d_0=\beta\log{n}$.
  Let $W\subseteq V(G)$ satisfy $h:=|W| \geq n/10$ and $\delta(G[W])\ge 2$, and assume that for every $v\in W$ we have $d(v,W)\ge \min\{ d(v), \beta \log n\}$.
  We select a random edge subgraph $H$ of $G[W]$ as follows.
  For each $v\in W$, if $d(v,W)\le d_0$ set $E(v)=E_G(v,W)$;
  otherwise, namely if $d(v,W)>d_0$, then set $E(v)$ to be a (uniformly) selected set of $d_0$ random edges of $G[W]$ which are incident to $v$.
  Let $H=(W,E_H)$ with $E_H=\bigcup_{v\in W} E(v)$.
  Observe that $|E(H)|\le h\cdot d_0\le \beta n\log{n}/10$.

  We now show that $H$ is, with positive probability, a (connected) $(h/4,2)$-expander.
  Taking $d=d_0$ and $m=\beta n/250$, and noting that $h\ge n/10\ge 4m$, it is enough to show that $H$ satisfies, with positive probability, Conditions 1--4 in \cref{lem:expander}.
  For the first condition, note that $\delta(H)\ge \min\{d_0,\delta(G[W])\}\ge 2$.
  The second condition holds as it holds for $G$ by \ref{P:smalldist} (since $d\le \log{n}/10$), and clearly also for every subgraph thereof.
  Similarly, noting that $5m=\beta n/50$, the third condition holds as it holds for $G$ by \ref{P:local:sparse}.

  We move on the prove the fourth condition of \cref{lem:expander}.
  Let $F_1,F_2\subseteq W$ with $|F_1|,|F_2|\ge m$.
  By \ref{P:pseudorandom} we know that $|E(F_1,F_2)|\ge cn\log{n}$ for $c=10^{-6}\beta^2$.
  For $u\in F_1$ for which $d_G(u,F_2)\ge 1$ let $\cA_u$ be the event that none of the edges of $D(u)$ is incident to a vertex of $F_2$.
  By the construction of $H$, if $d_G(u,W)\le d_0$ then $\pr(\cA_u)=0$.
  On the other hand, if $d_G(u,W)>d_0$ then, using \ref{P:maxdeg},
  \[
  \begin{aligned}
    \pr(\cA_u)
    &\le \binom{d_G(u,W)-d_G(u,F_2)}{d_0} / \binom{d_G(u,W)}{d_0}
    = \prod_{i=0}^{d_0-1} \frac{d_G(u,W)-d_G(u,F_2)-i}{d_G(u,W)-i}\\
    &\le \left( 1-\frac{d_G(u,F_2)}{d_G(u,W)} \right)^{d_0}
     \le \left( 1-\frac{d_G(u,F_2)}{10\log{n}} \right)^{d_0}
     \le \exp\left( -d_G(u,F_2)\cdot\beta/10 \right).
  \end{aligned}
  \]
  Note also that $\cA_u$ are independent for different $u\in F_1$.
  Thus,
  \[
    \pr(E_H(F_1,F_2)=\es)
    \le \exp\left(-\frac{\beta}{10}\sum_{u\in F_1} d_G(u,F_2)\right)
    = \exp\left( -\frac{\beta}{10} |E_G(F_1,F_2)| \right)
    \le e^{-c'n\log{n}}
  \]
  for $c'=10^{-7}\beta^3$.
  By taking the union bound over all at most $2^{2n}$ choices of $F_1,F_2$, we see that Condition~4 of \cref{lem:expander} holds \whp{}.

  Our next step is to show that $G[W]$ is Hamiltonian.
  Fix a subgraph $H$ of $G[W]$ which is a $(h/4,2)$-expander.
  To find a Hamilton cycle in $G[W]$ we define a sequence $H_0,H_1,\ldots,H_h$ of subgraphs of $G[W]$ as follows.
  Set $H_0=H$.
  For each $i\ge 0$, if $H_i$ is Hamiltonian then set $H_{i+1}=H_i$; otherwise, let $e_i$ be a booster of $H_i$ which is contained in $G[W]$.
  Note that such a booster is guaranteed to exist by \ref{P:boosters}, as $|E(H_i)|\le |E(H)|+i \le \beta n\log{n}/10+h \le \beta n\log{n}$.
  Evidently, one cannot add $h$ boosters to a graph on $h$ vertices sequentially without making it Hamiltonian, hence $H_h$ is a Hamiltonian subgraph of $G[W]$.

  Now, let $w\in W$, and let $P$ be a Hamilton path in $G[W]$ with $w$ being one of its endpoints.
  Let $R$ be the set of endpoints $y$ of Hamilton paths of $G[W]$ with endpoints $w$ and $y$.
  Evidently, as $G[W]$ is Hamiltonian, $R$ is not empty.
  Moreover, by \cref{lem:Posa} we have $|N_{G[W]}(R)|\le 2|R|-1$.
  Since $G[W]$ is a $(h/4,2)$-expander (since $H$ is such), it must be the case that $|R|>h/4\ge n/40$, so the assertion of the lemma holds.
\end{proof}

We are now ready to prove \cref{lem:prep}.

\begin{proof}[Proof of \cref{lem:prep}]
  Let $K,U_1,U_2$ be the disjoint subsets of $V=V(G)$ obtained in \cref{lem:partition}.
  Set $V^\star= U_1\cup U_2$ and $V'=V\sm V^\star$ (so $K\subseteq V'$ is of size $t$ and $D_1\subseteq K\subseteq D_{\le 2}$, hence \textbf{(a)} is satisfied).
  Note that \textbf{(b)} and \textbf{(c)} are also satisfied by \cref{lem:partition}.
  Let $K\subseteq X\subseteq V'$ satisfy $|X|\le n/2$ and let $w_1,w_2\in X\sm K$.
  Write $V''=V'\sm X$ and partition $V''=V''_1\cup V''_2$ as equally as possible.
  For $i=1,2$, let $W_i=V''_i\cup U_i$, and choose a neighbour $z_i$ of $w_i$ in $W_i$; this is possible since $d(w_i,U_i)\ge \gamma\log{n}/100$ by the condition in \cref{lem:partition}.
  Note that $|W_i|\ge n/5$ and for every $v\in W_i$ it holds that $d(v,W_i)\ge \min\{ d(v), \gamma\log{n}/100\}$, hence by \cref{lem:ham} there exists a set $R_i\subseteq W_i$ with $|R_i|\ge n/40$ such that for every $y\in R_i$ there is a Hamilton path spanning $W_i$ from $z_i$ to $y$.
  In view of \ref{P:pseudorandom}, there exists an edge $e$ between $R_1$ and $R_2$ with endpoints $y_i\in R_i$, say.
  For $i=1,2$, denote by $Q_{y_i}$ the Hamilton path between $w_i$ and $y_i$.
  We now construct a Hamilton path on $G[V(G)\sm X]$ as follows (as depicted in~\cref{fig:ham}):
  \[
    z_1 \xrightarrow{Q_{y_1}} y_1 \xrightarrow{e} y_2 \xrightarrow{Q_{y_2}} z_2,
  \]
  hence \textbf{(d)} is satisfied.
\end{proof}

\begin{figure*}[t]
  \captionsetup{width=0.879\textwidth,font=small}
  \centering
  \includegraphics{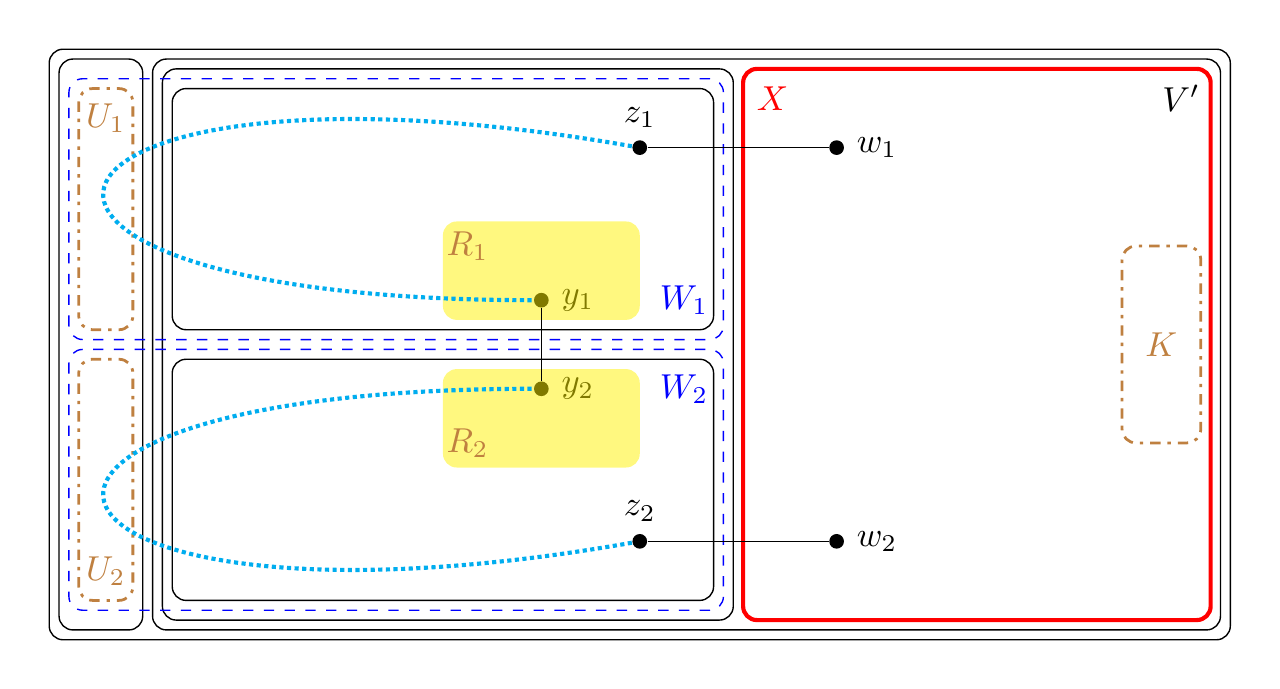}
  \caption{Visualisation of the proof of \cref{lem:prep}.}
  \label{fig:ham}
\end{figure*}

Let $V(G)=V^\star\cup V'$ be the partition obtained by \cref{lem:prep}, and let $K\subseteq V'$ be the set of size $t$ obtained by it.
The following lemma guarantees that $V'$ contains a copy of the KeyChain's ``comb'', which when put together with property \textbf{(d)} from \cref{lem:prep} guarantees the existence of a copy of $\KC(n,t,\ell)$ in $G$.

Write $K=\{x_1,\ldots,x_t\}$, and for each $i\in[t]$ let $w_i$ be an arbitrary neighbour of $x_i$ in $V'$ (there exist such neighbours due to the properties of $K,V'$, and they are distinct due to \ref{P:smalldist}).
Set $Q=\{w_1,\ldots,w_t\}$.
Recall the definitions of $a_j$ and $j_0$ from \cref{sec:intro}.

\begin{lemma}\label{lem:paths}
There is a sequence of paths $P_1,...,P_{t-1} \subseteq G[V']$ for which the following holds:
\begin{enumerate}
  \item The endpoints of $P_i$ are $\{w_i,w_{i+1}\}$ for all $1\le i \le t-1$;
	\item The length of $P_i$ is exactly $\ell$ for all $1\le i \le t-1$;
	\item $V\left( P_i \right) \cap V\left( P_{i+1} \right) = \left\{ w_{i+1} \right\}$ for all $1\le i < t-1$, and $V\left( P_i \right) \cap V\left( P_j \right) = \es$ for all $1\le i,j \le t-1$ such that $|i-j|>1$.
\end{enumerate}
\end{lemma}

\begin{proof}
  For $U\subseteq V^{\prime}$, set $N^{\prime}(U):= N_G(U) \cap V^{\prime}$ (and similarly $N^{\prime}(v):= N_G(v)\cap V^{\prime}$ for $v\in V^{\prime}$). By \ref{P:smalldist}, $Q \cap \Small = \es$, and $N^{\prime}(w_i) \cap N^{\prime}(w_j) = \es$ for all $i \neq j$.
  For each $1\leq i\leq t$ let $Y_i,Z_i \subseteq N^{\prime}(w_i)$ be arbitrary disjoint subsets of size $|Y_i| = |Z_i| = a_2$ (such sets exist, by the construction of $V^{\prime}$ in \cref{lem:prep}).
  We now construct the required paths sequentially.
  For $1\le i<t$ we assume that $P_1,P_2,...,P_{i-1}$ have already been constructed, and construct a path $P_i$ with the desired properties.
  Additionally, we construct $P_i$ to be such that its internal vertices do not belong to $K':=K\cup Q \cup \left( \bigcup_{k=1}^{t-1} \left( Y_k \cup Z_{k+1} \right) \right)$ (and, accordingly, assume that the internal vertices of $P_1,P_2,...,P_{i-1}$ do not belong to $Y_i, Z_{i+1}$).

  Set $S_2 := Y_i$, $T_2 := Z_{i+1}$.
  Now, for $j \leq \ell /2$, given the sets $S_2,...,S_{j-1},T_2,...,T_{j-1}\subseteq V^{\prime}$ we construct sets $S_j,T_j\subseteq V^{\prime}$ with the following properties:
  \begin{itemize}
  	\item $S_j \subseteq N^{\prime}\left( S_{j-1} \right)$, $T_j \subseteq N^{\prime}\left( T_{j-1} \right)$;
  	\item $|S_j| = |T_j| = a_j$;
  	\item $S_j \cap \left( \bigcup\limits_{k=1}^{i-1} V(P_k) \right) = T_j \cap \left( \bigcup\limits_{k=1}^{i-1} V(P_k) \right) = \es$;
  	\item $S_j \cap \left( \bigcup\limits_{k=2}^{j-1} \left( S_k \cup T_k \right) \right) = T_j \cap \left( \bigcup\limits_{k=2}^{j-1} \left( S_k \cup T_k \right) \right) = \es$;
  	\item $S_j \cap K' = T_j \cap K'= S_j \cap T_j = \es$.
  \end{itemize}

  We make the following observation, obtained from properties \ref{P:smallsets},\ref{P:local:sparse} and from the construction of $K,V^{\prime},V^*$ in \cref{lem:prep}:
  if $U \subseteq V^{\prime} \sm K$ is of size at most $10n/\log n$, then $|N^{\prime}(U)| \geq |U|\cdot \log n/30$.
  Indeed, assume otherwise, then
  \[
  |N(U)| \leq |N'(U)| + 200\gamma |U|\log n \leq \left( \frac{1}{30} + \frac{1}{50}\right) |U| \log n \leq |U|\log n/18.
  \]
  On the other hand, since $|U| \leq \frac{10n}{\log n} \leq \gamma n/5000$, by \ref{P:local:sparse}, $U$ spans at most $\gamma |U| \log n/1000$ edges, and since $U \cap \Small = \es$, we have
  \[
  e(U,V(G)\sm U) \ge \sum_{u\in U}d(u) - 2\cdot e(U) \geq |U|\log n/11.
  \]
  a contradiction to \ref{P:smallsets}.
  Therefore, since $|S_{j-1}| = |T_{j-1}| = a_{j-1} \leq a_{\ell /2 -1} < 10n / \log n$, we have
  \[
  	|N^{\prime}(S_{j-1})|,|N^{\prime}(T_{j-1})| \geq a_{j-1} \cdot \log n / 30.
  \]
  This inequality implies the existence of two disjoint subsets $S^{\prime} ,T^{\prime}$ of $N^{\prime}(S_{j-1}),N^{\prime}(T_{j-1})$, respectively, of size at least $a_{j-1}\cdot \log n/60$. In addition, recalling that $\ell = o(\log n)$, we get
  \[
  	\left| \left( \bigcup\limits_{k=1}^{i-1}V(P_i) \right) \cup \left( \bigcup\limits_{k=2}^{j-1} \left( S_k \cup T_k \right) \right)  \right| \leq i\cdot \ell + 2\cdot j \cdot a_{j-1} = o(a_{j-1} \cdot \log n) .
  \]
  We now wish to make sure that we can choose large enough subsets of $S',T'$ which do not intersect $K'$.
  To this end, note that by \ref{P:smalldist} $N^{\prime}(S_2) \cap K', N^{\prime}(T_2) \cap K' \subseteq Q$ and $|Q|=o(a_2\log{n})$,
  and for $j>3$, $|K'|=O(\log^2{n})=o(a_{j-1}\log{n})$, so for every $3\le j\le\ell/2$ we have
  \[
    |N'(S_{j-1})\cap K'|,|N'(T_{j-1})\cap K'| =o( a_{j-1}\cdot\log n).
  \]
  So overall we get
  \[
  	\left|
      \left( S^{\prime} \cup T^{\prime} \right)
      \cap \left(
             K' \cup \left( \bigcup\limits_{k=1}^{i-1}V(P_k) \right)
                \cup \left(
                       \bigcup\limits_{k=2}^{j-1} \left( S_k \cup T_k \right)
                     \right)
           \right)
    \right| = o(a_{j-1}\cdot \log n),
  \]
  which implies that there are subsets $S_j,T_j$ of $S^{\prime} , T^{\prime}$ with all the listed properties.
  Finally, observe that $|S_{\ell /2}| = |T_{\ell /2}| = a_{\ell /2}$, and therefore
  \[
  	10n / \log n \leq |S_{\ell /2}|, |T_{\ell /2}| \leq \frac{10n}{\log n} \cdot \left\lceil \frac{\log n}{100} \right\rceil \leq \frac{n}{9},
  \]
  and therefore, by \ref{P:intersect},
  \[
  |N^{\prime}(S_{\ell /2}) \cap N^{\prime}(T_{\ell /2})| \geq |N_G(S_{\ell /2}) \cap N_G(T_{\ell /2})| - |V^*| \geq \left( \frac{1}{9}-3\gamma \right) \cdot n \geq n/10,
  \]
  which implies that $N^{\prime}(S_{\ell /2}) \cap N^{\prime}(T_{\ell /2})$ contains a vertex that is not a member of $\left( \bigcup_{k=1}^{i-1}V(P_k) \right) \cup \left( \bigcup_{k=2}^{\ell /2-1} \left( S_k \cup T_k \right) \right)$.
  By the definitions of $S_2,...,S_{\ell /2},T_2,...,T_{\ell /2}$, this proves that there is a path $P_i$ of length $\ell$ between $w_i$ and $w_{i+1}$ with all our desired properties.
\end{proof}

This concludes the proof of \cref{lem:kc}.
Indeed, let $P=\bigcup_{i=1}^t P_i$ be the union of the paths we have found in \cref{lem:paths}, and let $X=P\cup \lbrace \{x_i,w_i\} \rbrace _{i=1}^t$ be the ``comb''.
By \cref{lem:prep}, there exist neighbours $z_1\sim w_1$ and $z_t\sim w_t$ outside the comb, and a Hamilton path in $G[V(G)\sm X]$ between $z_1$ and $z_t$.
The union of the comb, the edges $\{w_1,z_1\}$ and $\{w_t,z_t\}$ and the Hamilton path constitutes a copy of $\KC(n,t,\ell)$ in $G$ (see \cref{fig:hamkc}). \qed

\begin{figure*}[t]
  \captionsetup{width=0.879\textwidth,font=small}
  \centering
  \includegraphics{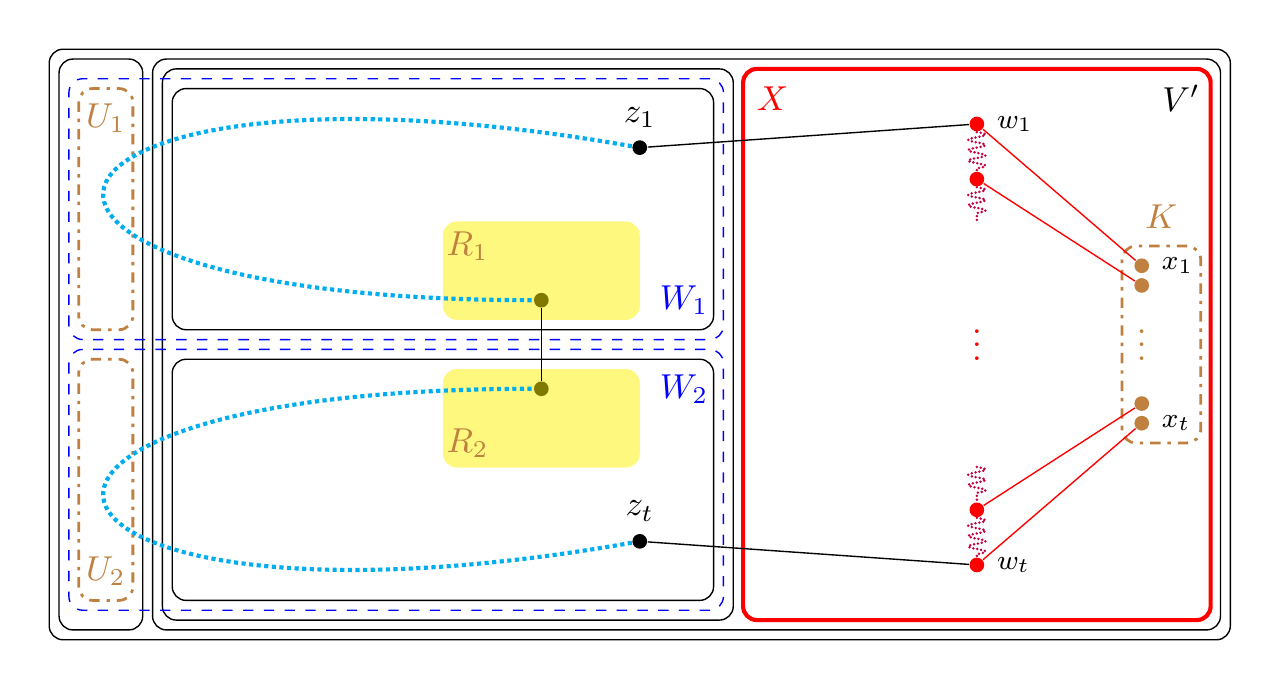}
  \caption{Visualisation of the proof of \cref{lem:kc}.}
  \label{fig:hamkc}
\end{figure*}

\section{Maximum common subgraph}\label{sec:mcs}
In this short section we prove \cref{prop:mcs}.
\begin{proof}[Proof of \cref{prop:mcs}]
  We may assume that $\eps>0$ is small enough.
  Let $\delta=\delta(\eps)>0$ to be chosen later and $p=n^{-1+\delta}$.
  Let $m=(1+\eps)n$, and let $\cA_m$ be the event that there exists a subgraph $H$ of $G_1$ with $m$ edges which is also a subgraph of $G_2$.
  By the union bound over the possible choices of $H$ and the permutations of the vertices of $G_2$, we obtain
  \[
    \pr(\cA_m) \le \binom{\binom{n}{2}}{m}\cdot n!\cdot p^{2m}
      \le \left(
            \left(\frac{enp^2}{2(1+\eps)} \right)^{1+\eps} \cdot n
          \right)^n
      \le \left(2\cdot n^{(-1+2\delta)(1+\eps)+1} \right)^n.
  \]
  Taking $\delta=\delta(\eps)>0$ small enough ($\delta \le \varepsilon /3$ suffices), the last term is vanishing.
\end{proof}

\bibliography{library}

\end{document}